\documentclass[11pt]{amsart}
\usepackage{geometry}                
\usepackage[parfill]{parskip}    
\usepackage{graphicx}
\usepackage{amssymb}
\usepackage{epstopdf}
\usepackage{multicol}
\usepackage{amsfonts}
\usepackage{pxfonts}
\usepackage{stmaryrd}
\usepackage{amsmath,amsthm,amssymb}
\usepackage[american]{babel}
\usepackage{xy}
\usepackage{hyperref}
\xyoption{all}

\newtheorem{thm}{Theorem}[section]
\newtheorem{lem}[thm]{Lemma}

\newtheorem{prop}[thm]{Proposition}

\theoremstyle{definition}

\newtheorem{defn}[thm]{Definition}

\newcommand{\C}{\mathbb{C}}
\newcommand{\R}{\mathbb{R}}

\newcommand{\T}{\mathbb{T}}

\newcommand{\LieG}{\mathfrak{g}}

\newcommand{\LieK}{\mathfrak{k}}

\newcommand{\LieA}{\mathfrak{a}}
\newcommand{\LieP}{\mathfrak{p}}

\newcommand{\LieM}{\mathfrak{m}}

\newcommand{\cO}{\mathcal{O}}

\title{The Higson-Mackey Analogy for Finite \\ Extensions of Complex Semisimple Groups}
\author{John R. Skukalek}

\begin{document}

\begin{abstract}
In the 1970's, George Mackey pointed out an analogy that exists between tempered representations of semisimple Lie groups and unitary representations of associated semidirect product Lie groups. More recently, Nigel Higson refined Mackey's analogy into a one-to-one correspondence for connected complex semisimple groups, and in doing so obtained a novel verification of the Baum-Connes conjecture with trivial coefficients for such groups. Here we extend Higson's results to any Lie group with finitely many connected components whose connected component of the identity is complex semisimple. Our methods include Mackey's description of unitary representations of group extensions involving projective unitary representations, as well as the notion of twisted crossed product $C^*$-algebra introduced independently by Green and Dang Ngoc.
\end{abstract}

\maketitle

\section{Introduction}

Let $G$ be a Lie group with finitely many connected components. Results for almost connected locally compact groups due to Chabert, Echterhoff, and Nest \cite{MR2010742} imply that $G$ satisfies the Baum-Connes conjecture for the trivial coefficient algebra $\C$. In the case of complex connected semisimple Lie groups, Higson \cite{MR2391803} verified the conjecture through a development of ideas originating from Mackey \cite{MR53:13478}. Mackey observed that the space of irreducible tempered representations of a connected semisimple Lie group closely resembles the space of irreducible unitary representations of an associated semidirect product Lie group. Given a maximal compact subgroup $K \subseteq G$ with respective Lie algebras $\LieK \subseteq \LieG$, we can form the semidirect product $K \ltimes V$ in which the action of $K$ on the quotient vector space $V = \LieG/\LieK$ is induced by the restriction to $K$ of the adjoint representation of $G$. The Lie groups $G$ and $K \ltimes V$ fit into a so-called smooth deformation of Lie groups $\{ G_t \ | \ 0 \leq t \leq 1 \}$ where $G_t = G$ for $t > 0$ and $G_0 = K \ltimes V$.

Connes \cite[Chapter 2, Section 10]{MR1303779} explained how the above construction results in a continuous field of reduced group $C^*$-algebras $\{ C^*_{\lambda}(G_t) \ | \ t \in [0,1] \}$ that can be used to reformulate the Baum-Connes conjecture with trivial coefficients for $G$ \cite{MR1292018}. Higson showed that in the case of connected complex semisimple Lie groups there exists a bijection between the set of equivalence classes of irreducible tempered representations of $G$ and the set of equivalence classes of irreducible unitary representations of $G_0$. When analyzed using $C^*$-algebras, Higson's bijection amounts to a verification of the Baum-Connes conjecture with trivial coefficients for connected complex semisimple Lie groups. This special case of the conjecture was first confirmed by Penington and Plymen \cite{MR724030}.

This article deals with the extension of Higson's results to almost connected Lie groups with complex semisimple connected component $G$. Such a group has finitely many connected components, and thus is an extension of $G$ by a finite group. We do not assume that the extension is a complex Lie group nor that it complies with any additional restrictions typically imposed upon reductive Lie groups which are allowed disconnectedness. It is our intention that an approach similar to the one explored here could be used as a means to extend future results for connected reductive Lie groups to disconnected reductive Lie groups, as well as almost connected Lie groups exhibiting behavior considered to be pathological within the realm of reductive groups. Below is an outline of our work.

Section 2 is devoted to describing Higson's results for connected complex semisimple Lie groups $G$. First we describe Higson's bijection between the tempered, or reduced, dual $\widehat{G}_{\lambda}$ of $G$ and the unitary dual $\widehat{G_0}$ of the associated semidirect product group. We then describe Higson's bijection at the level of the group $C^*$-algebras $C^*_{\lambda}(G)$ and $C^*(G_0)$, where one finds pairs of subquotients that are Morita equivalent to the same commutative $C^*$-algebra. An elaboration of this subquotient analysis to the continuous field $\{C^*_{\lambda}(G_t) \ | \ t \in [0,1] \}$ ultimately results in the verification of the Baum-Connes conjecture with trivial coefficients for $G$.

In Section 3 we extend Higson's bijection to the almost connected case using Mackey's normal subgroup analysis, also known simply as the Mackey machine \cite[Chapter 3]{MR0396826}. Here a role is played by projective representations of the finite group of connected components. This section is concerned purely with group representations and $C^*$-algebras do not appear. The explicit form of the bijection that we obtain is not used in the remainder of the paper, but is of independent interest.

In Section 4, following Higson's approach, we analyze our bijection using group $C^*$-algebras. We in fact reestablish the existence of such a bijection without considering projective representations and the finer details of the Mackey machine. In order to deal with group extensions that are not semidirect products, we make use of the notion of twisted crossed product $C^*$-algebra due to Green \cite{MR0493349} and Dang Ngoc \cite{dang}. We ultimately find pairs of subquotients that are Morita equivalent to the same twisted crossed product. These twisted crossed products are obtained from a twisted action involving the group of connected components on a direct sum of $C^*$-algebras, each summand of which is Morita equivalent to a commutative algebra appearing in Higson's analysis of the connected case. For disconnected groups, the subquotients need not be Morita equivalent to commutative $C^*$-algebras.

Finally in Section 5 we show how our results lead to a verification of the Baum-Connes conjecture for the almost connected Lie groups under consideration. The argument parallels that of Higon's in the connected case, utilizing the same basic properties of $K$-theory for group $C^*$-algebras. Such an approach to the Baum-Connes conjecture emphasizes group representation theory as opposed to the deeper aspects of $C^*$-algebra $K$-theory.

I would like to thank my dissertation advisor Nigel Higson, who has been an invaluable source of inspiration and guidance, and whose work serves as the foundation of this article. I would also like to thank Paul Baum for conversations both enlightening and entertaining.

\section{Higson's Results for Connected Complex Semisimple Groups}

Consider a semidirect product group $K \ltimes V$ in which $K$ is a compact group and $V$ is a locally compact abelian group. The Mackey machine \cite[Chapter 3]{MR0396826} describes all equivalence classes of unitary representations of such a semidirect product in terms of the characters $\chi : V \to \T = \{ z \in \C | \  |z|=1 \}$ of $V$ and irreducible unitary representations of the isotropy subgroups
\[ K_{\chi} = \{k \in K \ | \ \chi(k^{-1}vk) = \chi(v) \ \forall \ v \in V \}. \]
Given a unitary representation $\tau$ of $K_{\chi}$, we can trivially extend both $\tau$ and $\chi$ to the semidirect product $K_{\chi} \ltimes V$ and form the tensor product unitary representation $\tau \otimes \chi$. Finally, we can consider the induced unitary representation
\[ \text{Ind}_{K_{\chi} \ltimes V}^{K \ltimes V} \tau \otimes \chi. \]
Mackey showed that as $\chi$ varies over a set of representatives of the orbits of $\widehat{V}$ under the action of $K$, and $\tau$ varies over each $\widehat{K_{\chi}}$, every equivalence class in $\widehat{K \ltimes V}$ is represented once.

Mackey realized that there is a strong resemblance between representations of semisimple Lie groups and such semidirect products. Let $G$ be a connected semisimple Lie group with finite center and let $K$ be a maximal subgroup of $G$. The adjoint representation of $G$ induces an action of $K$ on the real vector space $V = \LieG/\LieK$ obtained as the quotient of the Lie algebras $\LieG$ and $\LieK$ of $G$ and $K$, respectively. We can then form the semidirect product Lie group $G_0 = K \ltimes V$. Mackey suggested that their exists a correspondence between equivalence classes of irreducible tempered representations of $G$ and equivalence classes of irreducible unitary representations of $G_0$. Mackey felt that such a correspondence would be measure-theoretic in nature, ignoring sets of irreducible representations with zero Plancherel measure.

Following Mackey's observations, Higson further developed this correspondence between the representation theories of $G$ and $G_0$ in the case that $G$ is additionally assumed to possess a complex structure. Higson discovered a bijection between the entire set of equivalence classes of irreducible tempered representations of $G$ and the entire set of equivalence classes of irreducible unitary representations of $G_0$. 

Higson's bijection utilizes the observation \cite[Lemma 2.2]{MR2391803} that in the complex semisimple case all the isotropy subgroups $K_{\chi}$ are connected. Choosing an Iwasawa decomposition $G=KAN$, the characters $\chi$ representing each orbit in $\widehat{V}$ can be selected in such a way such that the centralizer $M = Z_K(A)$ of $A$ in $K$ is a maximal torus in $K_{\chi}$. The Cartan-Weyl description of $\widehat{K_{\chi}}$ in terms of highest weights combined with Mackey's description of $\widehat{G_0}$ then reveals a bijection
\[ (\widehat{M} \times \widehat{A})/W \overset\cong\longrightarrow \widehat{G_0} \]
where $W = N_K(A)/M$ is the Weyl group of $G$, $N_K(A)$ denoting the normalizer of $A$ in $K$.

All equivalence classes of irreducible tempered representations of a connected complex semisimple Lie group $G$ are accounted for by the unitary principal series of $G$. Given a pair $(\sigma,\varphi)$ in which $\sigma \in \widehat{M}$ and $\varphi \in \widehat{A}$, we have the one-dimensional unitary representation $\sigma \otimes \varphi \otimes 1$ of the Borel subgroup $B=MAN$ involving the trivial representation of $N$. We can then form the induced representation
\[ \text{Ind}_B^G \ \sigma \otimes \varphi \otimes 1. \]
This procedure accounts for all equivalence classes of irreducible tempered representations of $G$. Moreover, two such representations  are equivalent if and only if the associated pairs in $\widehat{M} \times \widehat{A}$ are related by the action of the Weyl group. Thus we have a bijection
\[ (\widehat{M} \times \widehat{A})/W \overset\cong\longrightarrow \widehat{G}_{\lambda}, \]
where $\widehat{G}_{\lambda}$ denotes the tempered, or reduced, dual of $G$, $\lambda$ referring to the (left) regular representation of $G$ on $L^2(G)$. This bijection turns out to be a homeomorphism with respect to the Fell topology on $\widehat{G}_{\lambda}$ \cite{MR0146681}. On the other hand the bijection for $\widehat{G_0}$ cannot be, $\widehat{G_0}$ being a non-Hausdorff space.

Higson's bijection between $\widehat{G}_{\lambda}$ and $\widehat{G_0}$ results from the above bijections with $(\widehat{M} \times \widehat{A})/W$. Although not a homeomorphism, Higson's bijection has the property of preserving minimal $K$-types, which we describe below.

Each irreducible representation $\tau$ of $K$ has associated with it a finite set of characters of $M$. These are the weights of $\tau$ with respect to $M$, i.e. the characters of $M$ that occur in the restriction of $\tau$ to $K$. The Weyl group $W = N_K(M)/M$ acts on the weights of $\tau$. If we use the exponential map to identify each character of $M$ with a linear function $\LieM \to \R$, $\LieM$ denoting the Lie algebra of $M$, then we can partially order the set $\widehat{M}/W$ of $W$-orbits in $\widehat{M}$ as follows:
$W \cdot \{\sigma\} \leq W \cdot \{\sigma'\}$ if the convex hull of $W \cdot \{\sigma\}$ in $\LieM^*$ is a subset of the convex hull of $W \cdot \{\sigma'\}$. Then $\tau$ has a unique orbit of weights, called highest weights, that is the maximum of all orbits of weights of $\tau$. Each highest weight has the special property of occurring with multiplicity one. The correspondence between irreducible representations of $K$ and their highest weights establishes a bijection
\begin{equation*}\label{highest weights}
\widehat{M}/W \overset\cong\longrightarrow \widehat{K} 
\end{equation*}
In this way we are able to transfer the partial ordering of $\widehat{M}/W$ to $\widehat{K}$. It can be shown that the irreducible representations of $G$ and $G_0$ corresponding to $(\sigma,\varphi) \in \widehat{M} \times \widehat{A}$ each have among their $K$-types the irreducible representation $\tau_{\sigma}$ of $K$ with highest weight $\sigma \in \widehat{M}$. Furthermore, $\tau_{\sigma} < \tau$ for all other $K$-types $\tau$ of both representations, making $\tau_{\sigma}$ the unique minimal $K$-type of both representations. The theory of minimal $K$-types was developed by Vogan \cite{MR86j:22021} and plays an important role in describing admissible representations of semismple groups. Analogous with highest weights of representations of compact groups, minimal $K$-types always occur with multiplicity one. Uniqueness of minimal $K$-types is particular to complex semsimple groups, failing  for $SL(2,\R)$. The definition of minimal $K$-type given above must be carefully modified for the general semisimple scenario.

Higson's bijection and property of preserving minimal $K$-types can be summarized as follows.

\begin{thm}\cite[Theorem 2.4, Lemmas 5.2 and 5.6]{MR2391803}\label{higson bijection} 
Let $\pi_{\sigma,\varphi}$ and $\pi_{\sigma,\varphi}^0$ denote, respectively, the irreducible representations of $G$ and $G_0$ correpsonding to the point $(\sigma,\varphi) \in \widehat{M} \times \widehat{A}$. Let $\tau: \widehat{G}_{\lambda} \cup \widehat{G_0} \to \widehat{K}$ denote the map assigning minimal $K$-types. Then setting $\Phi(\pi_{\sigma,\varphi}) = \pi_{\sigma,\varphi}^0$
determines a bijection $\Phi : \widehat{G}_{\lambda} \overset\cong\longrightarrow \widehat{G_0}$
such that the diagram
\[ \xymatrix{\widehat{G}_{\lambda} \ar[rd]_{\tau} \ar[rr]^{\Phi} & & \widehat{G_0} \ar[ld]^{\tau} \\
 & \widehat{K} & } \]
commutes.
\end{thm}

As already mentioned, $\Phi$ is not a homeomorphism. However, Higson showed that $\Phi$ restricts to a homeomorphism between the subspaces of representations with a fixed minimal $K$-type $\tau$. If $\sigma$ is a highest weight of $\tau$, and $W_{\sigma}$ denotes the isotropy subgroup of $\sigma$ in $W$, these spaces of representations can be identified with the image of $\widehat{A}/W_{\sigma}$ in $(\widehat{M} \times \widehat{A})/W$ under the map $\varphi \mapsto (\sigma,\varphi)$. One can define (\cite[Definition 6.8]{MR2391803}) subquotients of $C^*_{\lambda}(G)$ and $C^*(G_0)$ whose spectra identify with the locally closed subsets of $\widehat{G}_{\lambda}$ and $\widehat{G_0}$ containing representations with minimal $K$-type $\tau$.

\begin{thm}\cite[Propositions 6.10 and 6.12]{MR2391803}\label{higson morita} Let $\tau$ be an irreducible representation with highest weight $\sigma$. Let $C_{\tau}$ and $C_{\tau}^0$ denote the subquotients of $C^*_{\lambda}(G)$ and $C^*(G_0)$ whose irreducible representations correspond to irreducible representations of $G$ and $G_0$ with minimal $K$-type $\tau$. Then $C_{\tau}$ and $C_{\tau}^0$ are Morita equivalent to the commutative $C^*$-algebra $C_0(\widehat{A}/W_{\sigma})$. In particular, the bijection $\Phi$ in Theorem~\ref{higson bijection} restricts to a homeomorphism between locally closed subsets of $\widehat{G}_{\lambda}$ and $\widehat{G_0}$ consisting of representations with minimal $K$-type $\tau$.
\end{thm}

Higson used the above analysis of subquotients to verify the Baum-Connes conjecture with trivial coefficients for $G$. For the moment, $G$ can be any Lie group with finitely many connected components. If we define $G_t = G$ for all $t \neq 0$, then the disjoint union
\[ \mathcal{G} = \bigsqcup_{t \in [0,1]} G_t \]
can be given a topology, in fact the structure of a smooth manifold with boundary \cite[Section 6.2]{MR2391803} \cite[Section 4]{MR2732053}. The algebra $C^{\infty}_c(\mathcal G)$ of smooth, compactly supported functions can be used to generate a $C^*$-algebra $C^*_{\lambda}(\mathcal G)$ of continuous sections of a continuous field of $C^*$-algebras $\{ C^*_{\lambda}(G_t) \ | \ t \in [0,1] \}$.
This continuous field is a deformation in the sense described by Connes and Higson \cite{MR1065438}, producing an asymptotic morphism
\[ \{ \mu_t : C^*_{\lambda}(G_0) \to C^*_{\lambda}(G) \ | \ t \in [0,1] \} \]
and a resulting homomorphism on $K$-theory groups
\[ \mu: K_*(C^*_{\lambda}(G_0)) \to K_*(C^*_{\lambda}(G)). \]
Connes showed \cite[Chapter 2, Section 10, Proposition 9]{MR1303779} that $\mu$ is an isomorphism if and only if $G$ satisfies the Baum-Connes conjecture with trivial coefficients.

Higson showed that $\mu$ is an isomorphism when $G$ is a connected complex semisimple Lie group by elaborating Theorem~\ref{higson morita} to an analysis of the continuous field. In extending this result to the almost connected case, we will make use of the following.

\begin{thm}\cite[Theorem 6.18]{MR2391803}\label{higson constant} Given $\sigma \in \widehat{M}$, $\varphi \in \widehat{A}$, and $t \in [0,1]$, denote by $\pi_{\sigma,\varphi}^t$ the corresponding irreducible representation of $G_t$. Then the map which assigns to $(\varphi,t)$ the irreducible representation of $C^*_{\lambda}(\mathcal G)$ defined by composing evaluation at $t$ with $\pi_{\sigma,\varphi}^t$ establishes a homeomorphism from $\widehat{A}/W_{\sigma} \times [0,1]$ onto its image in the spectrum of $C^*_{\lambda}(\mathcal G)$.
\end{thm}

Theorem~\ref{higson constant} amounts to the fact that the continuous field is assembled from constant fields with fiber $C_0(\widehat{A}/W_{\sigma})$ via operations that are well-behaved with respect to $K$-theory. From here it can be seen that $\mu$ is an isomorphism.

\section{Extension of Higson's Bijection using the Mackey Machine}

We will now establish an explicit extension of Higson's bijection $\Phi$ in Theorem~\ref{higson bijection} to the almost connected case. It is explicit in the sense that it involves an actual construction of the representations involved. The tool for accomplishing this is the general version of the Mackey machine \cite[Chapter 3, Section 10]{MR0396826}. We will in fact obtain sharper results than what we need for the $C^*$-algebra analysis that is to come.

Throughout the remainder of the article we let $\tilde G$ denote a Lie group with finitely many connected components whose connected component of the identity is a complex semisimple group denoted by $G$. The component group $\tilde G/G$ is a finite group, which we denote by $F$. We shall preserve all of the notation used for $G$ in the previous section. Thus $K$ shall denote a maximal compact subgroup of $G$. Let
\[ \tilde K = N_{\tilde G}(K)\]
denote the normalizer of $K$ in $\tilde G$. It is not difficult to see that $\tilde K$ is a maximal compact subgroup of $\tilde G$. Since $G$ is semisimple, the Cartan decomposition of $\LieG$ implies that the normalizer of $K$ in $G$ is equal to $K$, so that $\tilde K \cap G = K$. The uniqueness of the maximal compact subgroup $K$ up to conjugation in $G$ implies that $\tilde K$ has nonempty intersection with every connected component of $\tilde G$. More generally, according to work of Borel \cite{MR1661166}, every Lie group with finitely many connected components has a maximal compact subgroup unique up to conjugation that has nonempty intersection with every connected component. The inclusion map $\tilde K \to \tilde G$ induces an isomorphism $\tilde K/K \overset\cong\longrightarrow \tilde G/G = F$.

We identify the Lie algebras of $\tilde G$ and $\tilde K$ with the Lie algebras of their connected components, $\LieG$ and $\LieK$. We form the semidirect product Lie group
\[ \tilde G_0 = \tilde K \ltimes V, \quad V = \LieG/\LieK. \]
Observe that
\[ \tilde G_0/G_0 \cong \tilde K/K \cong \tilde G/G = F. \]
In this way we are able to identify the groups of connected components of $\tilde G$, $\tilde K$, and $\tilde G_0$ with the same finite group $F$.

The component group $F$ acts on the unitary duals of $G$, $K$, and $G_0$ according to the manner in which any quotient group acts on the unitary dual of a closed normal subgroup.

\begin{defn}
Given $f \in F$, let $x \in \tilde G$ and $y \in \tilde K$ be such that $xG = yG = f$. Let $\pi$, $\tau$, and $\pi^0$ be unitary representations of $G$, $K$, and $G_0$, respectively. We define $f \cdot \pi = x \cdot \pi$, $f \cdot \tau = y \cdot \tau$, and $f \cdot \pi^0 = y \cdot \pi^0$ to be the unitary representations defined by
\[ [f \cdot \pi](g) = [x \cdot \pi](g) = \pi(x^{-1}gx) \]
\[ [f \cdot \tau](k) = [y \cdot \tau](k) = \tau(y^{-1}ky) \]
\[ [f \cdot \pi^0](g) = [y \cdot \pi^0](g_0) = \pi^0(y^{-1}g_0y) \]
for all $g \in G$, $k \in K$, and $g_0 \in G_0$.
\end{defn}

From this we obtain actions of $F$ on the unitary duals $\widehat{G}$, $\widehat{K}$, and $\widehat{G_0}$. Since $x \in \tilde G$ fixes the left regular representation of $G$, the action of $F$ on $\widehat{G}$ restricts to an action of $F$ on the reduced, or tempered, dual $\widehat{G}_{\lambda}$.

\begin{prop}\label{equivariance}
Higson's bijection $\Phi$ and the assignment $\tau$ of minimal $K$-types in Theorem~\ref{higson bijection} are equivariant with respect to the action of $F$.
\end{prop}

\begin{proof}
Consider the principal series representation $\pi_{\sigma,\varphi} \in \widehat{G}_{\lambda}$. Let $y \in \tilde K$. If $\LieP$ denotes, as usual, the orthogonal complement of $\LieK$ in $\LieG$ with respect to the Killing form, then the Lie algebra $\LieA$ of $A$ is a maximal abelian subspace of $\LieP$, which is mapped by $Ad(y)$, $Ad$ denoting the adjoint representation, onto another maximal abelian subspace of $\LieP$. Since all such subspaces of $\LieP$ are conjugate via $K$ \cite{MR1920389}, there exists $k \in K$ such that $Ad(yk)\LieA = \LieA$. Thus we may assume that $f \in F$ is represented by $y \in \tilde K$ that normalizes $A$, and therefore also normalizes $M$, the centralizer of $A$ in $K$. The automorphism $g \mapsto ygy^{-1}$ of $G$ provides an equivalence between $y \cdot \pi_{\sigma,\varphi}$ and the induced representation $\text{Ind}_{MAyNy^{-1}} y \cdot \sigma \otimes y \cdot \varphi \otimes 1$. This latter representation is equivalent to $\text{Ind}_{MAN} y \cdot \sigma \otimes y \cdot \varphi \otimes 1$. Thus $f \cdot \pi_{\sigma,\varphi}$ is the principal series representation $\pi_{y \cdot \sigma,y \cdot \varphi}$ in which $\sigma$ and $\varphi$ have been acted upon by $y$.

Consider $\pi^0_{\sigma,\varphi} = \Phi(\pi_{\sigma,\varphi}) \in \widehat{G_0}$ corresponding to $\pi_{\sigma,\varphi}$. Thus we extend $\varphi$ to a character of $V$ defined by $v \mapsto \varphi(\text{exp}(P(v)))$ where $P: V \to V$ is the orthgononal projection onto $\LieA$ with respect to the Killing form. In order to observe that $\Phi(f \cdot \pi_{\sigma,\varphi}) = f \cdot \Phi(\pi_{\sigma,\varphi})$, we need to know that $y$ maps this character of $V$ to the character of $V$ obtained by extending $y \cdot \varphi$. This follows from the fact that $y$ normalizes $A$ and therefore commutes with $P$. Now observe that $yK_{\varphi}y^{-1} = K_{y \cdot \varphi}$, where we have identified $\varphi$ with its extension to $V$. If $\tau$ denotes the irreducible representation of the isotropy subgroup $K_{\varphi}$ with highest weight $\sigma$, then $y \cdot \sigma$ is a weight of $y \cdot \tau \in \widehat{K_{y \cdot \varphi}}$ with the same highest weight vector $v$, since $(y \cdot \tau)(m)v = \tau(y^{-1}my)v = \sigma(y^{-1}my)v = (y \cdot \sigma)(m)v$ for all $m \in M$. Since $Ad(y)$ preserves the Killing form, $y \cdot \sigma$ must be the highest weight of $y \cdot \tau$. Thus $y \cdot (\pi^0_{\sigma,\varphi})$ is equivalent, using the automorphism $(k,v) \mapsto y(k,v)y^{-1}$ of $G_0$, to the induced representation $\text{Ind}_{K_{y \cdot \varphi} \ltimes V}^{G_0} y \cdot \tau \otimes y \cdot \varphi$, which is precisely $\Phi(f \cdot \pi_{\sigma,\varphi})$. This also verifies that the minimal $K$-type of $f \cdot \pi_{\sigma,\varphi}$ and $f \cdot \pi^0_{\sigma,\varphi}$ is $f \cdot \tau$ where $\tau$ is the minimal $K$-type of $\pi_{\sigma,\varphi}$ and $\pi_{\sigma,\varphi}^0$.
\end{proof}

We will use the Mackey machine to describe irreducible representations of $\tilde G$ and $\tilde G_0$ in terms of irreducible representations of $G$ and $G_0$. Mackey's construction proceeds as follows. Let $\pi : G \to U(H_{\pi})$ be an irreducible unitary representation of $G$ in the Hilbert space $H_{\pi}$. Consider the isotropy subgroup
\[ \tilde G_{\pi} = \{ x \in \tilde G \ | \ x \cdot \pi \cong \pi \}. \]
We define isotropy subgroups $F_{\pi}$ and $\tilde K_{\pi}$ of $F$ and $\tilde K$ similarly. Note that $\tilde G_{\pi}/G$ and $\tilde K_{\pi}/K$ can both be identified with $F_{\pi}$.
For each $x \in G_{\pi}$, there exists a unitary operator
\[ U_x^{\pi} : H_{\pi} \to H_{\pi} \]
satisfying
\[ \pi(x^{-1}gx) = U_x^{\pi *}\pi(g)U_x^{\pi} \]
for all $g \in G$.
In other words, $U_x^{\pi}$ intertwines $\pi$ and $x \cdot \pi$. Since $\pi$ is irreducible, $U_x^{\pi}$ is uniquely determined up to multiplication by an element of the circle group $\T = \{ z \in \C \ | \ |z| = 1 \}$. Note that for each $g \in G$, $U_g^{\pi} = \pi(g)$ is one such intertwining operator.
Consider a section $s : F_{\pi} \to \tilde G_{\pi}$, that is, a map that chooses from every coset $f \in F_{\pi}$ a representative $s(f) \in \tilde G_{\pi}$, so that $s(f)G = f$ for all $f \in F_{\pi}$. Assume that $s$ selects the identity element from $G$. Now for each $f \in F_{\pi}$, select an intertwining operator $U_{s(f)}^{\pi}$ as above. We assume that $U_{s(G)}^{\pi}$, $s(G)$ being the identity element of $G$ is the identity map. Now, for each $f \in F_{\pi}$ and $g \in G$, we define $U_{s(f)g}^{\pi} = U^{\pi}_{s(f)}\pi(g)$. In this way we obtain a map $U^{\pi} : \tilde G_{\pi} \to U(H_{\pi})$ whose restriction to $G$ is $\pi$. Given $x,y \in \tilde G_{\pi}$, the product $U_x^{\pi}U_y^{\pi}$ of the corresponding intertwining operators intertwines $\pi$ and $xy \cdot \pi$. Note that for $g,h \in G$, we have $U_{gh}^{\pi} = U_g^{\pi}U_h^{\pi}$. In this way, $\pi \in \widehat{G}$, along with our choice of section $s$ and intertwining operators $U_{s(f)}^{\pi}$, determine a function
\[ \omega_{\pi} : F_{\pi} \times F_{\pi} \to \T  \]
such that
\[ U_{xy}^{\pi} = \omega_{\pi}(xG,yG)U_x^{\pi}U_y^{\pi} \]
for all $x,y \in \tilde G_{\pi}$.
Such a function is known as a \textit{Schur multiplier} on $F_{\pi}$. Our choice of intertwining operators has effectively extended $\pi$ to an irreducible \textit{projective} unitary representation $U^{\pi}$ of $\tilde G_{\pi}$. Let $\rho$ be an irreducible projective representation of $F_{\pi}$ with Schur multiplier $\omega_{\pi}^{-1}$, i.e. $\omega_{\pi}^{-1}(xG,yG) = \omega_{\pi}(xG,yG)^{-1}$ for all $x,y \in \tilde G_{\pi}$. We can trivially extend $\rho$ to $\tilde G_{\pi}$ and $\omega_{\pi}^{-1}$ to a multiplier of $\tilde G_{\pi}$. The tensor product
$U^{\pi} \otimes \rho$ is then an ordinary unitary representation of $\tilde G_{\pi}$, and we may form the induced representation
\[ \text{Ind}_{\tilde G_{\pi}}^{\tilde G} U^{\pi} \otimes \rho. \]
It is an irreducible unitary representation of $\tilde G$. With $\pi$ fixed, its equivalence class depends only on the equivalence class of the projective representation $\rho$, while inequivalent projective representations produce inequivalent representations of $\tilde G$. An induced representation involving a representation of $G$ equivalent to $f \cdot \pi$ for some $f \in F$ is equivalent to a representation involving $\pi$. Induced representations involving representations of $G$ that belong to distinct $F$-orbits in $\widehat{G}$ are inequivalent.

Under favorable circumstances, the above procedure produces all equivalence classes of irreducible unitary representations of $\tilde G$. In Mackey's terms, such favorable circumstances occur when $G$ is \textit{regularly embedded} in $\tilde G$ \cite[Chapter 3, Section 8]{MR0396826}. The fact that we are in such circumstances can be seen from the fact that $G$, or rather the $C^*$-algebra $C^*(G)$, is type I, or equivalently, in Dixmier's terminology, postliminal \cite[9.1]{MR0458185}, while the quotient $\tilde G/G$ is finite. More generally, it suffices to know that $G$ is type I and $\tilde G$ second-countable, while the orbits in $\widehat{G}$ under the action of $\tilde G/G$ are locally closed, that is, relatively open subsets of their closures \cite[Theorem 1]{glimm}, \cite[4.4.5]{MR0458185}.

In our situation in which $\tilde G/G$ is finite, the restriction to $G$ of any irreducible unitary representation of $\tilde G$ is a direct sum of irreducible unitary representations of $G$ belonging to the same $F$-orbit in $\widehat{G}$, and each representation appears with the same multiplicity. The original representation of $\tilde G$ can be obtained as an induced representation as above beginning with any one of these representations of $G$.

The regular representation of any locally compact group can be obtained as the induced representation of the regular representation of any closed subsgroup. It follows then \cite[Theorem 4.3]{MR0159898} that Mackey's description of $\widehat{\tilde G}$ produces an element of $\widehat{\tilde G}_{\lambda}$ if and only if $\pi \in \widehat{G}_{\lambda}$.

Proposition~\ref{equivariance} implies in particular that corresponding representations $\pi$ and $\Phi(\pi)$ have identical isotropy subgroups of $F$. We wish to compare the Schur multipliers arising from such a pair of representations.

\begin{lem}\label{section} Let $\pi = \pi_{\sigma,\varphi} \in \widehat{G}_{\lambda}$. Let $\tau$ be the minimal $K$-type of $\pi$ with highest weight $\sigma \in \widehat{M}$. Then there exists a section $s: F_{\pi} \to \tilde{K}_{\pi}$ such that for every $f \in F_{\pi}$, $s(f)$ does the following:

\begin{enumerate}
\item normalizes $M$ and $A$ while fixing $\sigma$ and $\varphi$.
\item fixes the minimal $K$-type of $\tau$.
\item fixes the character of $V$ corresponding to $\varphi$.
\item normalizes the isotropy subgroup in $K$ of the character in (3) and fixes its irreducible representation with highest weight $\sigma$.
\end{enumerate}
\end{lem}

\begin{proof} In fact, (1) implies (2), (3), and (4). Let $f \in F_{\pi}$. As in the proof of Proposition~\ref{equivariance}, we may represent $f$ with $y \in \tilde K$ that normalizes $M$ and $A$. We then see that $f \cdot \pi$ is the principal series representation $\pi_{y \cdot \sigma,y \cdot \varphi}$. This being equivalent to $\pi_{\sigma,\varphi}$, there must exist $k \in K$, representing an element of the Weyl group, such that $yk \cdot \sigma = \sigma$ and $yk \cdot \varphi = \varphi$. We then define $s(f) = yk$. Since $s(f) \tau \in \widehat{K}$ has highest weight $s(f) \cdot \sigma = \sigma$, we have $s(f) \cdot \tau = \tau$. In the proof of Proposition~\ref{equivariance}, we made use of the fact that the map $\widehat{A} \to \widehat{V}$ given by extending characters is equivariant with respect to the action of the normalizer of $A$ in $\tilde K$. Thus $s(f)$ fixes the character of $V$ corresponding to $\varphi$. Continuing to denote this character by $\varphi$, we have $s(f)K_{\varphi}s(f)^{-1} = K_{s(f)\cdot \varphi} = K_{\varphi}$. If $\tau$ now denotes the irreducible representation of $K_{\varphi}$ with highest weight $\sigma$, then $s(f) \cdot \tau$ has highest weight $s(f) \cdot \sigma = \sigma$, so $s(f) \cdot \tau = \tau$.
\end{proof}

\begin{prop}\label{multiplier} Let $\pi \in \widehat{G}_{\lambda}$ and define a map $u: F_{\pi} \times F_{\pi} \to K$ by
\[ u(f_1,f_2) = s(f_1f_2)^{-1}s(f_1)s(f_2) \]
where $s: F_{\pi} \to \tilde{K}_{\pi}$ is as in Lemma~\ref{section}.
Let $\tau$ be the minimal $K$-type of $\pi$ and $v$ a highest weight vector of $\tau$. Then $v$ is an eigenvector of $u(f_1,f_2)$ for all $f_1,f_2 \in F_{\pi}$. The function $\omega: F_{\pi} \times F_{\pi} \to \T$ defined by
\[ \tau(u(f_1,f_2))v = \omega(f_1,f_2)v \]
is the Schur multuplier arising from some extension of $\pi$ to a projective representation $U^{\pi}$ of $\tilde G_{\pi}$.
\end{prop}

\begin{proof}
Let $f \in F_{\pi}$ and $y = s(f)$. Consider a unitary operator $U^{\pi}_y : H_{\pi} \to H_{\pi}$ that intertwines $y \cdot \pi$ with $\pi$, so that $U^{\pi}_y\pi(y^{-1}gy) = \pi(g)U^{\pi}_y$ for all $g \in G$. This holds in particular for all $g \in K$, so that $U^{\pi}_y$ intertwines the representation $k \mapsto \pi(y^{-1}ky)$ of $K$ on the $\tau$-isotypical component $H_{\pi}^{\tau}$ of $\pi$ with $\pi$ restricted to $K$. Thus $U^{\pi}_y$ maps $H_{\pi}^{\tau}$ into $H_{\pi}^{y \cdot \tau}$, the $y \cdot \tau$-isotypical component of $\pi$. Since $y \cdot \tau \cong \tau$, we have that $U^{\pi}_y(H_{\pi}^{\tau}) = H_{\pi}^{\tau}$. Let $v \in H_{\pi}^{\tau}$ be a highest weight vector associated to $\sigma \in \widehat{M}$, so that $v \neq 0$ and $\pi(m)v = \sigma(m)v$ for all $m \in M$. Recall that $\tau$ occurs with multiplicity one in $\pi$, and $\sigma$ with multiplicity one in $\tau$, so that $v$ is unique up to a scalar multiple. Observe that $\pi(m)U^{\pi}_yv = U^{\pi}_y\pi(y^{-1}my)v = U^{\pi}_y\sigma(y^{-1}my)v = \sigma(y^{-1}my)U^{\pi}_yv = \sigma(m)U^{\pi}_yv$ for all $m \in M$, since $y$ normalizes $M$ and fixes $\sigma$. Since $U^{\pi}_yv \in H_{\pi}^{\tau}$, we see that $U^{\pi}_yv$ must be a scalar multiple of $v$. Thus we may scale $U^{\pi}_y$ so that it is the identity on the $\sigma$-highest weight subspace of the $\tau$-isotypical component of $\pi$.

Now let $f_1,f_2 \in F_{\pi}$, and observe that $s(f_1f_2)$ and $s(f_1)s(f_2)$ belong to the same connected component of $\tilde K$. Hence there exists a unique $u(f_1,f_2) \in K$ such that $s(f_1f_2)u = s(f_1)s(f_2)$. Like $s(f_1)$, $s(f_2)$, and $s(f_1f_2)$, $u(f_1,f_2)$ normalizes $M$ and $A$ while fixing $\sigma$ and $\varphi$, and thus, as above, $\pi(u(f_1,f_2))$ must act as a scalar $\omega(f_1,f_2) \in \T$ on the $\sigma$-isotypical component of the $\tau$-isotypical component of $\pi$. Thus we have $U^{\pi}_{s(f_1)s(f_2)}v = U^{\pi}_{s(f_1f_2)u(f_1,f_2)}v = U^{\pi}_{s(f_1f_2)}U^{\pi}_{u(f_1,f_2)}v =\omega(f_1,f_2)\pi(s(f_1f_2))v = \omega(f_1,f_2)v = \omega(f_1,f_2)\pi(s(f_1))\pi(s(f_2))v$.
\end{proof}

Lemma~\ref{section} and Proposition~\ref{multiplier} apply to $G_0$ in the same way that they apply to $G$. Since $\pi \in \widehat{G}_{\lambda}$ and $\Phi(\pi) \in \widehat{G_0}$ have the same minimal $K$-type and determine the same isotropy subgroup of $F_{\pi} = F_{\Phi(\pi)}$, $\pi$ and $\Phi(\pi)$ may be extended to projective unitary representations giving rise to identical Schur multipliers. The Mackey machine descriptions of $\widehat{\tilde G}_{\lambda}$ and $\widehat{\tilde G_0}$ now lead us to the following result.

\begin{thm}\label{extended bijection} Associate to each irreducible unitary representation $\tilde{\pi}= \text{Ind}_{\tilde{G}_{\pi}}^{\tilde G} U^{\pi} \otimes \rho$ of $\tilde{G}$ the irreducible unitary representation $\tilde{\Phi}(\tilde{\pi}) = \text{Ind}_{\tilde{G}_{0\Phi(\pi)}  }^{\tilde{G}_0} U^{\Phi(\pi)} \otimes \rho$ of $\tilde{G}_0$. Then we obtain a bijection
\[ \tilde{\Phi} : \widehat{\tilde{G}}_{\lambda} \overset\cong\longrightarrow \widehat{\tilde{G}_0} \]
with the property that an irreducible tempered representation $\pi'$ of $G$ occurs in the restriction of $\tilde{\pi}_{\rho}$ to $G$ if and only 
$\Phi(\pi')$ occurs in the restriction of $\tilde{\Phi}(\tilde{\pi}_{\rho})$ to $G_0$.
\end{thm}
\begin{proof}
The fact that $\tilde{\Phi}$ is well-defined is a consequence of the completeness of the Mackey machine description of $\widehat{\tilde G}_{\lambda}$ along with the $F$-equivariance of Higson's bijection $\Phi$ (Proposition~\ref{equivariance}) and the conditions under which the Mackey machine produces equivalent representations. Combinined with injectivity of $\Phi$ we obtain injectivity of $\tilde \Phi$. The surjectivity of $\tilde \Phi$ follows from the completeness of the Mackey machine description of $\widehat{\tilde{G}_0}$ along with the surjectivity of Higson's bijection and the fact that $\pi$ and $\Phi(\pi)$ give rise to identical Schur multipliers (Proposition~\ref{multiplier}).
\end{proof}

\section{Analysis of Subquotients using Twisted Crossed Products}

We now will carry out a $C^*$-algebraic analysis of the bijection in Theorem~\ref{extended bijection}. Actually we will reprove the existence of such a bijection, so that Theorem~\ref{extended bijection} provides more detailed information than is necessary for what follows. As in Higson's treatment of the connected case, we will concern ourselves with subquotients of $C^*_{\lambda}(\tilde G)$ and $C^*(\tilde G_0)$ determined by irreducible representations of $K$. In the almost connected case, a pair of subquotients will be associated with each orbit $\cO$ in $\widehat{K}$ under the action of the finite component group $F$. To simplify our discussion, we shall refer primarily to the group $G$ and point out that it can be replaced everywhere with the group $G_0$.

Recall that if we restrict an irreducible representation of $\tilde G$ to $G$, it decomposes into a direct sum of members of $\widehat{G}_{\lambda}$, which constitute an orbit under the action of $F$. The minimal $K$-types of these irreducible representations of $G$ constitute an orbit in $\widehat{K}$ under the action of $F$.

Recall the partial ordering of $\widehat{K}$ induced by that of $\widehat{M}/W$ and the bijection obtained via highest weights. Suppose that $y \in \tilde K$ normalizes $M$. Identifying characters of $M$ with linear functionals $\LieM \to \R$, the linear automorphism of $\LieM^*$ obtained by differentiating the action of $y$ on $M$ maps the weights of an irreducible representation $\tau$ of $K$ onto the weights of the irreducible representation $y \cdot \tau$ of $K$. Moreover, it maps the convex hull of the weights of $\tau$ onto the convex hull of the weights of $y \cdot \tau$. Thus if $\tau < \tau'$, then $y \cdot \tau < y \cdot \tau'$. It is not possible to have $\tau < y \cdot \tau$, since the orbits in $\widehat{K}$ under the action of $\tilde K$ are finite.

\begin{defn}
Given two $F$-orbits $\cO$ and $\cO'$ in $\widehat{K}$, we write $\cO \leq \cO'$ if there exists $(\tau,\tau') \in \cO \times \cO'$ such that $\tau \leq \tau'$.
\end{defn}

If $\cO \leq \cO'$, then for every $\tau \in \cO$, there exists $\tau' \in \cO'$ such that $\tau \leq \tau'$. Thus if $\cO \leq \cO'$ and $\cO' \leq \cO''$, then there exists $(\tau,\tau',\tau'') \in \cO \times \cO' \times \cO''$ such that $\tau \leq \tau'$ and $\tau' \leq \tau''$. Hence $\tau' \leq \tau''$ and $\cO \leq \cO''$. If this occurs when $\cO'' = \cO$, then $\tau \leq \tau''$ implies that $\tau = \tau''$, $\tau' \leq \tau$, $\tau = \tau'$ and $\cO = \cO'$. Thus we obtain a partial ordering of $\widehat{K}/\tilde K$ that agrees with the usual partial ordering when $\tilde K$ is connected. We shall use this partial ordering to define the subquotients we are interested in.

For each $\tau \in \widehat{K}$, we have the character $\chi_{\tau} : K \to \C$ of $\tau$ defined by $\chi_{\tau}(k) = \text{trace}(\tau(k))$ for all $k \in K$. Let $\pi \in \widehat{G}_{\lambda}$. The isotypical component $H_{\pi}^{\tau}$ of $\pi$ is related to $\chi_{\tau}$ in the following way. The representation of $C^*_{\lambda}(G)$ determined by $\pi$ extends uniquely to a representation of the multiplier algebra $M(C^*_{\lambda}(G))$ of $C^*_{\lambda}(G)$.
There is a homomorphism from the $C^*$-algebra $C^*(K)$ of $K$ into the multiplier algebra of $C^*_{\lambda}(G)$ defined for continuous functions $\psi \in C(K)$ and $\varphi \in C_c(G)$ via the convolution integral
\[ (\psi \varphi)(g) = \int_K \psi(k)\varphi(k^{-1}g) \ dk. \]
With the Haar measure on $K$ satisfying $\int_K 1 \ dk =1$, the function $p_{\tau} : K \to \C$ by defined by $p_{\tau}(k) = \text{dim}(\tau)\chi_{\tau}(k^{-1})$ defines in this way  a projection in the multiplier algebra of $C^*_{\lambda}(G)$, and the operator $\pi(p_{\tau}): H_{\pi} \to H_{\pi}$ is the orthogonal projection onto $H_{\pi}^{\tau}$. We can replace $G$ with $G_0$, $\tilde G$, or $\tilde G_0$, so that $p_{\tau}$ may also be considered a projection in the multiplier algebras of $C^*(G_0)$, $C^*_{\lambda}(\tilde G)$, and $C^*(\tilde G_0)$.

\begin{defn}\label{orbit multipliers}
For each orbit $\cO \subseteq \widehat{K}$, denote by $p_{\cO}$ the projection in $M(C^*_{\lambda}(G))$ defined by
\[ p_{\cO} = \sum_{\tau \in \cO} p_{\tau}. \] We define ideals $A_{\cO}$ and $B_{\cO}$ in $C^*_{\lambda}(G)$, with $B_{\cO}$ an ideal in $A_{\cO}$ as follows:
\[ A_{\cO} = \sum_{\cO' \leq \cO} C^*_{\lambda}(G)p_{\cO}C^*_{\lambda}(G) \]
\[ B_{\cO} =  \sum_{\cO' < \cO} C^*_{\lambda}(G)p_{\cO}C^*_{\lambda}(G). \]
Denote by $C_{\cO}$ the subquotient
\[ C_{\cO} = A_{\cO}/B_{\cO}. \]
In the same way, we define ideals $\tilde A_{\cO}$ and $\tilde B_{\cO}$ in $C^*_{\lambda}(\tilde G)$ and a subquotient $\tilde C_{\cO}$, using $\tilde G$ in place of $G$.
\end{defn}

We shall be interested in correspondence between locally closed subsets of the spectrum of a $C^*$-algebra and the spectra of subquotients \cite[3.2]{MR0458185}. Combining it with the correspondence between the spectrum of $C^*_{\lambda}$ and $\widehat{G}_{\lambda}$, we obtain the following.

\begin{prop}\label{spectra}
The spectrum of $C_{\cO}$ can be identified with the locally closed subset of $\widehat{G}_{\lambda}$ consisting of representations whose minimal $K$-type belongs to $\cO$. The spectrum of $\tilde C_{\cO}$ can be identified with the locally closed subset of $\tilde \pi \in \widehat{G}_{\lambda}$ for which the restriction of $\tilde \pi$ to $G$ contains an element of the spectrum of $C_{\cO}$. In this case, every element of the spectrum of $C_{\cO}$ occurs in the restriction of $\tilde \pi$ to $G$.\end{prop}
\begin{proof}
The spectrum of $C^*_{\lambda}(\tilde G)p_{\cO}C^*_{\lambda}(\tilde G)$ consists of all $\pi$ in the spectrum of $C^*_{\lambda}(\tilde G)$ such that $\pi(p_{\cO}) \neq 0$, and $\pi(p_{\cO}) \neq 0$ if and only if $\pi(p_{\tau}) \neq 0$ for some $\tau \in \cO$, i.e. $\tau$ is a $K$-type of $\pi$. The spectrum of $\tilde A_{\cO}$ thus consists of all $\pi \in \widehat{\tilde G}_{\lambda}$ for which the orbit $F \cdot \tau$ of some $K$-type $\tau$ of $\pi$ satisfies $F \cdot \tau \leq \cO$. The spectrum of $\tilde B_{\cO}$ consists of all $\pi \in \widehat{\tilde G}_{\lambda}$ for which the orbit $F \cdot \tau$ of some $K$-type $\tau$ of $\pi$ satisfies $F \cdot \tau < \cO$. Thus the spectrum of $\tilde C_{\cO}$ consists of all $\pi \in \widehat{\tilde G}_{\lambda}$ for which the orbit $F \cdot \tau$ of some $K$-type $\tau$ of $\pi$ equals $\cO$, and no orbit $F \cdot \tau'$ of any other $K$-type $\tau'$ satisfies $F \cdot \tau' < \cO$. In the case that $\tilde G = G$, such a $\tau$ would be the unique minimal $K$-type of $\pi$. In the general case, each member of $\cO$ is the unique minimal $K$-type of some irreducible constituent of the restriction of $\pi$ to $G$.
\end{proof}

We wish to analyze $\tilde C_{\cO}$ up to Morita equivalence. Regarding $p_{\cO}$ as an element of the multiplier algebra of $\tilde C_{\cO}$, we see from Proposition~\ref{spectra} that $\tilde C_{\cO}$ is generated by $p_{\cO}$, since every irreducible representation $\pi$ of $\tilde C_{\cO}$ satisfies $\pi(\tilde C_{\cO}p_{\cO}\tilde C_{\cO}) \neq \{0\}$. Thus $\tilde C_{\cO}$ is Morita equivalent to the subalgebra $p_{\cO}\tilde C_{\cO}p_{\cO}$. The same, of course, is true with $C_{\cO}$ in place of $\tilde C_{\cO}$. 

We will use the notion of twisted crossed product $C^*$-algebra due to Green \cite{MR0493349} and Dang Ngoc \cite{dang} to relate $C^*_{\lambda}(\tilde G)$ to $C^*_{\lambda}(G)$ and, in turn, $\tilde C_{\cO}$ to $C_{\cO}$. Ordinary, as opposed to twisted, crossed products suffice in the event that $\tilde G$ is isomorphic to the semidirect product $F \ltimes G$. Let $A$ be any $C^*$-algebra equipped with a continuous action $\alpha: \tilde G \to \text{Aut}(\tilde G)$ of $\tilde G$ by automorphisms of $A$. Let $\sigma : G \to UM(A)$ be a strictly continuous \cite[p. 34]{MR2288954} homomorphism from $G$ into the unitary group of the multiplier algebra of $A$ with the following properties:
\[ \alpha_g(a) = \sigma(k)a\sigma(k)^*, \quad \sigma(xgx^{-1}) = \alpha_x(\sigma(k)), \quad \forall g \in G, a \in A, x \in \tilde G. \]
Note that the automorphism $\alpha_x$ has been extended to the multiplier algebra of $A$.
We refer to $\sigma$ as a \textit{twisting map} for the action $\alpha$, and the pair $(\alpha,\sigma)$ as a \textit{twisted action} of $(\tilde G,G)$. The \textit{twisted crossed product} $C^*$-algebra $(\tilde G,G) \ltimes_{\alpha,\sigma} A$ is the quotient of the ordinary crossed product $\tilde G \ltimes_{\alpha} A$ by the ideal corresponding to irreducible covariant representations $(U,\pi)$ of $(\tilde G,A)$ that \textit{preserve} the twisting map, in the sense that $U_g = \pi( \sigma(g) )$ for all $g \in G$. We can describe $(\tilde G,G) \ltimes A$ can be described more explicitly as follows, quite similar to the manner in which $\tilde G \ltimes A$ is defined, only replacing integrals over $\tilde G$ with integrals, or in our case sums, over $\tilde G/G$.
Let $C(\tilde G,G,A)$ denote the set of continuous functions $\varphi : \tilde G \to A$ such that
\[ \varphi(gx) = \varphi(x)\sigma(g)^* \quad \forall g \in G, x \in \tilde G. \]
We define multiplication in $C(\tilde G,G,A)$ by
\[ (\varphi \psi)(y) = \sum_{xG \in \tilde G/G} \varphi(x) \alpha_x[\psi(x^{-1}y)] \quad \forall \ \varphi,\psi \in C(\tilde G,G,A). \]
Involution is given by
\[ \varphi^*(x) = \alpha_x[\varphi(x^{-1})^*] \quad \forall \ \varphi \in C(\tilde G,G,A). \]
A covariant representation $(U,\pi)$ of $(\tilde G,A)$ that preserves the twisting map determines a $\ast$-representation of $C(\tilde G,G,A)$ as follows:
\[ (U \ltimes \pi)(\varphi) = \sum_{xG \in \tilde G/G} \pi[\varphi(x)]U_x \ dx \quad \forall \ \varphi \in C(\tilde G,G,A). \]
To obtain the twisted crossed product $(\tilde G,G) \ltimes A$, we complete $C(\tilde G,G,A)$ with respect to the norm
\[ ||\varphi|| = \sup_{(U,\pi) \ \text{that preserves} \ \sigma} (U \ltimes \pi)(\varphi). \]
We can also consider the twisted crossed product $(\tilde K,K) \ltimes A$ where the action and twisting map have been restricted to the compact groups $\tilde K$ and $K$, respectively. Every covariant representation $(U,\pi)$ of $(\tilde G, A)$ that preserves the twisting map restricts to a covariant representation $(U|_{\tilde K}, \pi)$ of $(\tilde K,A)$ that preserves the twisting map. Conversely, every covariant representation $(U,\pi)$ of $(\tilde K,A)$ that preserves the twisting map can be extended to a covariant representation of $(\tilde G,A)$ that preserves the twisting map by defining
\[ U_x = U_y\pi(\sigma(g)), \quad \forall \ x=yg \in \tilde G, y \in \tilde K, g \in G. \]
This is independent of the representation of $x$ as the product $yk$ because $\tilde K \cap G = K$, whereby $\tilde G/G \cong \tilde K/K$. It thus follows that restricting function from $\tilde G$ to $\tilde K$ determines an isomorphism
\[ (\tilde G,G) \ltimes A \overset\cong\longrightarrow (\tilde K,K) \ltimes A. \]
In this sense the twisted crossed product obtained from a twisted action of $(\tilde G,G)$ depends only on the quotient $\tilde G/G$. This justifies referring to such a twisted action as a twisted action of $F$, which reduces to an ordinary action of $F$ when the twisting map is trivial.

If $\tilde G$ is isomorphic to the semidirect product $F \ltimes G$, $F = \tilde G/G$, then there is an isomorphism between $C^*(\tilde G)$ and the ordinary crossed product $F \ltimes C^*(G)$. More generally, $C^*(G)$ is isomorphic to a twisted crossed product $(\tilde G,G) \ltimes C^*(G)$. The action of $\tilde G$ on $C^*(G)$ is induced by the action $g \mapsto xgx^{-1}$ of $\tilde G$ on $G$ by conjugation, and the twisting map is induced by the action $g \mapsto g'g$ of $G$ on itself by left translation. As above, we shall restrict this twisted action to $(\tilde K,K)$, obtaining an isomorphism $(\tilde G,G) \ltimes C^*(G)$ and $(\tilde K,K) \ltimes C^*(G)$. Below we describe the isomorphism between $C^(\tilde G)$ and $(\tilde K,K) \ltimes C^*(G)$ and show that, $\tilde K/K$ being finite, this isomorphism induces an isomorphism involving the reduced $C^*$-algebras as well as the subquotients that we wish to analyze.

 \begin{lem}\label{twisted crossed} The map $\Theta$ which associates to $\varphi \in C_c(\tilde G)$ the function $\Theta(\varphi) : \tilde K \to C_c(G)$ defined by $[\Theta(\varphi)](y) : g \mapsto \varphi(gy)$, $(y,g) \in \tilde K \times G$, extends to an isomorphism
\[ \Theta : C^*_{\lambda}(\tilde G) \overset{\cong}\longrightarrow (\tilde K,K) \ltimes C^*_{\lambda}(G). \]
For each orbit $\cO \subseteq \widehat{K}$, this isomorphism induces an isomorphism of subquotients
\[ \tilde C_{\cO} \overset\cong\longrightarrow (\tilde K,K) \ltimes C_{\cO} \]
that maps the subalgebra $p_{\cO}\tilde C_{\cO}p_{\cO}$ of $\tilde C_{\cO}$ onto the subalgebra $(\tilde K,K) \ltimes p_{\cO}C_{\cO}p_{\cO}$ of $(\tilde K,K) \ltimes C_{\cO}$. 

\end{lem}
\begin{proof}
The fact that $\Theta$ induces an isomorphism between $C^*(\tilde G)$ and $(\tilde G,G) \ltimes C^*(G)$ is a special case of \cite[Proposition 7.28]{MR2288954}. As established there, an irreducible unitary representation $\pi$ of $G$ corresponds under this isomorphism to the covariant representation $(\pi, \pi_G)$. Recall that the restriction to $G$ of any irreducible unitary representation $\pi$ of $\tilde G$ decomposes into a finite direct sum of irreducible representations of $G$, and $\pi$ belongs to $\widehat{\tilde G}_{\lambda}$ if and only if each of these irreducible representation of $G$ belong to $\widehat{G}_{\lambda}$. It follows that the isomorphism between $C^*(\tilde G)$ and $(\tilde G,G) \ltimes C^*(G)$ maps the kernel of the regular representation of $\tilde G$ onto the ideal in $(\tilde G,G) \ltimes C^*(G)$ that can be identified with $(\tilde G,G) \ltimes \ker \lambda$, $\lambda$ denoting the left regular representation of $G$. Thus $\Theta$ induces an isomorphism between $C^*_{\lambda}(\tilde G)$ and $[(\tilde G,G) \ltimes C^*(G)]/[(\tilde G,G) \ltimes \ker \lambda] \cong (\tilde G,G) \ltimes C^*_{\lambda}(G)$. As mentioned, that $\Theta$ defines an isomorphism with $(\tilde K,K) \ltimes C^*_{\lambda}(G)$ follows from the fact that $\tilde G/G \cong \tilde K/K$.
 
We observe that for all $\psi \in C(K)$ and $\varphi \in C_c(\tilde G)$, the invariance of Haar measure on $K$ with respect to the action of $\tilde K$ implies that
\[ [\Theta(\psi \varphi)](y) = (y^{-1} \cdot \psi)[\Theta(\varphi)(y)] \]
for all $y \in \tilde K$. Since $p_{\cO}$ is fixed by the action of $\tilde K$, we find that $\Theta$ maps the ideal $C^*_{\lambda}(\tilde G)p_{\cO}C^*_{\lambda}(\tilde G)$ onto the ideal $(\tilde K,K) \ltimes C^*_{\lambda}(G)p_{\cO}C^*_{\lambda}(G)$. Thus $\Theta$ maps $\tilde A_{\cO}$ onto $(\tilde K,K) \ltimes A_{\cO}$, inducing an isomorphism from $\tilde C_{\cO}$ to $(\tilde K,K) \ltimes C_{\cO}$. In the same way we find that $\Theta$ maps
$p_{\cO}\tilde C_{\cO}p_{\cO}$ onto $(\tilde K,K) \ltimes p_{\cO}C_{\cO}p_{\cO}$.
\end{proof}

When $\tau$ is the minimal $K$-type of $\pi$, $\tau$ has multiplicity one in $\pi$, which allows us to make the following definition.

\begin{defn}\label{transform} For each equivalence class $\tau \in \widehat{K}$, let us choose a finite-dimensional Hilbert space $V_{\tau}$ that represents $\tau$. Denote by $\text{End}(V_{\tau})$ the $C^*$-algebra of endomorphisms of $V_{\tau}$, i.e. linear maps $V_{\tau} \to V_{\tau}$.
Denote by $X_{\tau}$ the locally closed subset of $\widehat{G}_{\lambda}$ consisting of representations whose minimal $K$-type is $\tau$.
Given $\varphi \in C_c(G)$, define a function $\widehat{\varphi}_{\tau} : X_{\tau} \to \text{End}(V_{\tau})$
so that we obtain a commutative diagram
\[ \xymatrix{V_{\tau} \ar[d]_{U_{\pi}}\ar[r]^{\widehat{\varphi}_{\tau}(\pi)} & V_{\tau} \\
H_{\pi}^{\tau} \ar[r]^{\pi(p_{\tau}\varphi p_{\tau})} & H_{\pi}^{\tau} \ar[u]_{U_{\pi}^*}} \]
in which $U_{\pi}: V_{\tau} \to H_{\pi}^{\tau}$ is any unitary operator intertwining $\tau$ with the restriction of $\pi$ to $K$.
\end{defn}

\begin{prop}\label{direct sum}
Let $\cO$ be an orbit in $\widehat{K}$ under the action of $F$. Then the map $\varphi \mapsto \sum_{\tau \in \cO} \widehat{\varphi}_{\tau}$ establishes an isomorphism
\[ p_{\cO}C_{\cO}p_{\cO} \overset\cong\longrightarrow \bigoplus_{\tau \in \cO} C_0(X_{\tau}, \text{End}(V_{\tau})). \]
\end{prop}
\begin{proof}
Let $\tau,\tau' \in \widehat{K}$, $\tau \neq \tau'$. Let $\varphi \in C^*_{\lambda}(G)$ and let $\pi$ be an irreducible representation of $p_{\cO}C_{\cO}p_{\cO}$. Since $\tau$ and $\tau'$ cannot both be the minimal $K$-type of $\pi$, we have that $\pi(p_{\tau}\varphi p_{\tau'})=0$. Thus $p_{\tau}\varphi p_{\tau'}=0$ in $p_{\cO}C_{\cO}p_{\cO}$. Since $p_{\cO} = \sum_{\tau \in \cO} p_{\tau}$ is the unit of $p_{\cO}C_{\cO}p_{\cO}$, we find that $p_{\cO}C_{\cO}p_{\cO} = \oplus_{\tau \in \cO} p_{\tau}C_{\cO}p_{\tau}$. Using the notation $A^K$ to mean the subalgebra of $A$ consisting of all elements of $A$ fixed by each element of $K$, consider
\[ [p_{\tau}C_{\cO}p_{\tau}]^K = p_{\tau}C_{\cO}^Kp_{\tau}. \]
According to the Peter-Weyl Theorem, the group $C^*$-algebra $C^*(K)$ is isomorphic to the direct sum of the matrix algebras $\text{End}(V_{\tau})$. The aforementioned $\ast$-homomorphism $C^*(K) \to M(C^*_{\lambda}(G))$ then allows us to consider the sequence of maps
\[ \textmd{End}(V_{\tau}) \to C^*(K) \to M(C^*_{\lambda}(G)) \to M(p_{\tau}C_{\cO}p_{\tau}). \]
In this way we define a $\ast$-homomorphism  define a $\ast$-homomorphism $\text{End}(V_{\tau}) \to M(p_{\tau}C_{\cO}p_{\tau})$, which is observed to be injective. Thus we way consider $\text{End}(V_{\tau})$ to be a subalgebra of the multiplier algebra of $p_{\tau}C_{\cO}p_{\tau}$. We then find that the commutator of $\textmd{End}(V_{\tau})$ in $M(p_{\tau}C_{\cO}p_{\tau})$ is precisely $p_{\tau}C_{\cO}^Kp_{\tau}$.
It follows that multiplication induces an isomorphism
\[ p_{\tau}C_{\cO}^Kp_{\tau} \otimes \textmd{End}(V_{\tau}) \overset\cong\longrightarrow p_{\tau}C_{\cO}p_{\tau}. \]
If $q \in \textmd{End}(V_{\tau})$ is any rank one projection, which we identify with a multiplier of $p_{\tau}C_{\cO}p_{\tau}$ and $\pi$ is an irreducible representation of $p_{\tau}C_{\cO}p_{\tau}$, the subspace $\pi(q)H_{\pi}$ of $H_{\pi}$ is invariant under $p_{\tau}C_{\cO}^Kp_{\tau}$, and every irreducible representation of $p_{\tau}C_{\cO}^Kp_{\tau}$ is of this form. Since $\tau$ occurs with multiplicity one in $\pi$, we find that the projection $\pi(p)$ also has rank one, its range coinciding with the image of $T(V_{\tau})$ under all intertwining operators from $V_{\tau}$ to $H_{\pi}$. Thus every irreducible representation of $p_{\tau}C_{\cO}^Kp_{\tau}$ is one-dimensional, and so $p_{\tau}C_{\cO}^Kp_{\tau}$ must be commutative. The spectrum of $p_{\tau}C_{\cO}^Kp_{\tau}$ can be identified with the locally closed subset $X_{\tau}$ of $\widehat{G}_{\lambda}$, so that we obtain an isomorphism
\[ p_{\tau}C_{\cO}p_{\tau} \overset\cong\longrightarrow C_0(X_{\tau}) \otimes \textmd{End}(V_{\tau}) = C_0(X_{\tau}, \textmd{End}(V_{\tau})). \]
One can check that the isomorphism
\[ p_{\cO}C_{\cO}p_{\cO} = \bigoplus_{\tau \in \cO} p_{\tau}C_{\cO}p_{\tau} \overset\simeq\longrightarrow  \bigoplus_{\tau \in \cO} C_0(X_{\tau}, \textmd{End}(V_{\tau})) \]
thus obtained coincides with the asserted formula.
\end{proof}

Recall the twisted action of $(\tilde K,K)$ on $p_{\cO}C_{\cO}p_{\cO}$. We now describe the corresponding twisted action on the direct sum in Proposition~\ref{direct sum}.

\begin{defn}\label{twisted action}
Let $y \in \tilde{K}$ and $\tau \in \widehat{K}$. Denote by $\tau'$ the equivalence class of the representation $y \cdot \tau$. For each $T \in \text{End}(V_{\tau})$, define $y \cdot T \in \text{End}(V_{\tau'})$ using the commutative diagram
\[ \xymatrix{ V_{\tau'} \ar[d]_{U_y}\ar[r]^{y \cdot T} & V_{\tau'} \\
V_{\tau} \ar[r]^T & V_{\tau} \ar[u]_{U_y^*} } \]
in which $U_y$ is any unitary operator intertwining $y \cdot \tau$ with $\tau'$.
Next, let $f \in \oplus_{\tau \in \cO} C_0(X_{\tau}, \text{End}(V_{\tau}))$ with components $f_{\tau} \in C_0(X_{\tau}, \text{End}(V_{\tau})$. Define
\[ [\alpha_y(f)]_{\tau}(\pi) = y \cdot f_{y^{-1} \cdot \tau}(y^{-1} \cdot \pi). \]
Finally, define, for each $k \in K$, $\sigma(k)f : X_{\tau} \to \text{End}(V_{\tau}))$ by
\[ [\sigma(k)f](\pi) = \tau(k)f(\pi). \]
\end{defn}

Let us say that a homomorphism $\phi : A \to A'$ of $C^*$-algebras with twisted actions $(\alpha,\sigma)$ and $(\alpha',\sigma')$ of $(\tilde K,K)$ is \textit{$(\tilde K,K)$-equivariant} if $\phi[\alpha_y(a)] = \alpha'_y[\phi(a)]$ for all $y \in \tilde K$ and $a \in A$, and $\phi[\sigma(k)] = \sigma'(k)$ for all $k \in K$. Such a homormophism naturally induces a homomorphism $(\tilde K,K) \ltimes A \to (\tilde K,K) \ltimes A'$ of twisted crossed products, which is an isomorphism if and only if $\phi$ is an isomorphism.

\begin{prop}\label{equivariant}
The pair $(\alpha, \sigma)$ defines a twisted action of $(\tilde K,K)$ on $\bigoplus_{\tau \in \cO} C_0(X_{\tau}, \text{End}(V_{\tau}))$ such that the isomorphism in Proposition~\ref{direct sum} is $(\tilde K,K)$-equivariant. 
\end{prop}
\begin{proof} It is straightforward to see that we have a strictly continuous homomorphism
\[ \sigma: K \to UM(\bigoplus_{\tau \in \cO} C_0(X_{\tau}, \textmd{End}(V_{\tau}))) \]
from $K$ into the unitary group in the multiplier algebra of the direct sum.
Observe that for $k \in K$, $\tau(k)$ intertwines $\tau$ with $k \cdot \tau$. Thus for $f \in C_0(X_{\tau}, \text{End}(V_{\tau}))$ and $\pi \in X_{\tau}$, we have
\[ [\alpha_k(f)](\pi) = k \cdot f(k^{-1} \cdot \pi) = k \cdot f(\pi) = \tau(k)f(\pi)\tau(k)^* = [\sigma(k)f\sigma(k)^*](\pi). \]
Thus $\alpha_k(f) = \sigma(k)f\sigma(k)^*$. With $y \in \tilde K$, we have
\[ [\sigma(yky^{-1})f](\pi) = \tau(yky^{-1})f(\pi) = [yky^{-1} \cdot f(\pi)]\tau(yky^{-1}) = \alpha_{yky^{-1}}(yky^{-1} \cdot \pi)\tau(yky^{-1}). \]

\[ [\alpha_y(\sigma(k))](f) = \alpha_y(\sigma(k) \alpha_y^{-1}(f) ) \]

\[ [\alpha_y(\sigma(k))](f)(\pi) = y \cdot \tau(k)[\alpha_y^{-1}(f)(y^{-1} \cdot \pi)] = \tau(y^{-1}ky) y \cdot[\alpha_y^{-1}(f)(y^{-1} \cdot \pi)] = \sigma(y^{-1}ky)f(\pi). \]
Thus $\alpha_y(\sigma(k)) = \sigma(y^{-1}ky)$, and so $(\alpha,\sigma)$ indeed defines a twisted action of $(\tilde K,K)$.

To see that the isomorphism in Proposition~\ref{direct sum} is $(\tilde K,K)$-equivariant,
\[ \widehat{(y \cdot \varphi)}_{\tau}(\pi) = U_{\pi}^*\pi(p_{\tau} (y \cdot \varphi) p_{\tau})U_{\pi} = U_{\pi}^*\pi(y \cdot (p_{y^{-1} \cdot \tau}\varphi p_{y^{-1} \cdot \tau})U_{\pi} = U_{\pi}^*(y^{-1} \cdot \pi)(p_{y^{-1} \cdot \tau}\varphi p_{y^{-1} \cdot \tau})U_{\pi}. \]
Thus
\[ y^{-1} \cdot \widehat{(y \cdot \varphi)}_{\tau}(\pi) = (U_{\pi}U_{y^{-1}})^*(y^{-1} \cdot \pi)(p_{y^{-1} \cdot \tau}\varphi p_{y^{-1} \cdot \tau})U_{\pi}U_{y^{-1}}. \]
The unitary operator
\[ U_{\pi}U_{y^{-1}}: V_{y^{-1} \cdot \tau} \longrightarrow H_{\pi}^{\tau} = H_{y^{-1} \cdot \pi}^{y^{-1} \cdot \tau} \]
intertwines $y^{-1} \cdot \tau$ with $(y^{-1} \cdot \pi)|_K$, and so
\[ y^{-1} \cdot \widehat{(y \cdot \varphi)}_{\tau}(\pi) = \widehat{\varphi}_{y^{-1} \cdot \tau}(y^{-1} \cdot \pi). \]
Thus
\[ \widehat{(y \cdot \varphi)}_{\tau}(\pi) = y \cdot \widehat{\varphi}_{y^{-1} \cdot \tau}(y^{-1} \cdot \pi) = [\alpha_y(\widehat{\varphi})]_{\tau}(\pi), \]
and so
\[ \widehat{ y \cdot \varphi } = \alpha_y(\widehat{\varphi}). \]
Finally, $\gamma: K \to UM(C^*_{\lambda}(G))$ denoting as before the twisting map for the twisted action of $(\tilde K,K)$ on $C^*_{\lambda}(G)$, we have
\[ \widehat{(\gamma(k)\varphi)}_{\tau}(\pi) = U_{\pi}^*\pi(p_{\tau}\gamma(k)\varphi p_{\tau})U_{\pi} = U_{\pi}^*\pi(\gamma(k) p_{\tau}\varphi p_{\tau})U_{\pi}  \]
since $\gamma(k)$ commutes with $p_{\tau}$. Since $\pi(\gamma(k)) = \pi(k)$, we have
\[ \widehat{(\gamma(k)\varphi)}_{\tau}(\pi) = U_{\pi}^*\pi(k) \pi(p_{\tau}\varphi p_{\tau})U_{\pi} = \tau(k)U_{\pi}^*\pi(p_{\tau}\varphi p_{\tau})U_{\pi} = \sigma(k)\widehat{\varphi}_{\tau}(\pi). \]
Thus $\widehat{\gamma(k)} = \sigma(k)$ so that $\varphi \mapsto \widehat{\varphi}$ is $(\tilde K,K)$-equivariant.

\end{proof}

In the same way that we defined the subquotient $\tilde C_{\cO}$ of $C^*_{\lambda}(\tilde G)$, we can define a subquotient $\tilde C_{\cO}^0$ of $C^*(\tilde G_0)$. Using the fact that the bijection $\widehat{G}_{\lambda} \overset\cong\rightarrow \widehat{G_0}$ is $\tilde K$-equivariant and restricts to homeomorphisms $X_{\tau} \overset\cong\rightarrow X_{\tau}^0$, we then obtain the following extension of Theorem~\ref{higson morita}.

\begin{thm}\label{extended morita}
For every orbit $\cO \subseteq \widehat{K}$, the subquotients $\tilde C_{\cO}$ and $\tilde C_{\cO}^0$ are each Morita equivalent to the twisted crossed product
\[ (\tilde K, K) \ltimes_{\alpha,\sigma} \bigoplus_{\tau \in \cO} C_0(X_{\tau},\text{End}(V_{\tau})). \]
\end{thm}
\qed

\section{Analysis of the Continuous Field}

In the final section we shall extend Higson's analysis of the continuous field $\{ C^*_{\lambda}(G) \ | \ t \in [0,1] \}$ for a complex connected semisimple group $G$ to finite extensions of $G$. In doing so, we will have verified the Baum-Connes conjecture with trivial coefficients for such almost connected groups.

We form a family of Lie groups $\{ \tilde G_t \ | \ t \in [0,1] \}$ in which $\tilde G_t = \tilde G$ for $t > 0$ and form the
disjoint union
\[ \mathcal{\tilde G} = \bigsqcup_{t \in [0,1]} \tilde G_t. \]
We define $\mathcal G$ similarly by replacing $\tilde G$ with $G$. As mentioned, these sets have a topological structure, indeed the structure of a smooth manifold with boundary \cite[Section 6.2]{MR2391803}. Following Higson's approach, we aim to establish a version of Theorem~\ref{extended morita} that applies to the continuous field $\{C^*_{\lambda}(\tilde G_t) \ | \ t \in [0,1] \}$.
We denote by $C^*_{\lambda}(\mathcal {\tilde G})$ and $C^*_{\lambda}(\mathcal G)$, respectively, the $C^*$-algebras of continuous sections of the continuous fields for $\tilde G$ and $G$. Each irreducible representation of $C^*_{\lambda}(\mathcal{\tilde G})$ is obtained by composing an irreducible representation of one of the fibers $C^*_{\lambda}(\tilde G_t)$ with the homomorphism $C^*_{\lambda}(\mathcal{\tilde G}) \to C^*_{\lambda}(\tilde G_t)$ given by evaluating sections at $t \in [0,1]$. In this way we can identify the spectra of this $C^*$-algebra with the disjoint union
\[ \bigsqcup_{t \in [0,1]} \widehat{(\tilde G_t)}_{\lambda}. \]

\begin{defn}
For each orbit $\cO$ in $\widehat{K}$ under the action of $F$, let $p_{\cO}$ denote the multiplier of $\mathcal{\tilde C}$ defined fiber-wise using the multipliers $p_{\cO} = \sum_{\tau \in \cO} p_{\tau}$ from Definition~\ref{orbit multipliers}. Now define subquotients $\mathcal{\tilde C}_{\cO}$ and $\mathcal{C}$ of $C^*_{\lambda}(\mathcal{\tilde G})$ and $C^*_{\lambda}(\mathcal G)$, respectively, 
 in the same way that we defined subquotients $\tilde C_{\cO}$ and $C_{\cO}$ of $C^*_{\lambda}(\tilde G)$ and $C^*_{\lambda}(G)$.
\end{defn}

The spectrum of $\mathcal{\tilde C}_{\cO}$ can be identified as a set with
\[ \{ \pi \in \widehat{(\tilde G_t)}_{\lambda} \ | \ t \in [0,1], \text{the minimal $K$-type of $\pi$ belongs to $\cO$} \}. \]

For each $t \in [0,1]$, we have the isomorphism $\Theta_t : C^*_{\lambda}(\tilde G_t) \to (\tilde K,K) \ltimes C^*_{\lambda}(G_t)$ from Lemma~\ref{twisted crossed}. Thus we can replace the continuous field $\{ C^*_{\lambda}(\tilde G_t) \ | \ t \in [0,1] \}$ with a continuous field $\{ (\tilde K,K) \ltimes C^*_{\lambda}(G_t) \ | \ t \in [0,1] \}$ whose sections are obtained using the above isomorphisms. Suppose that $\varphi$ is a continuous, compactly supported function $\tilde{\mathcal G} \to \C$. With $\varphi_t$ denoting its restriction to $\tilde G_t$, we may apply $\Theta_t$ in order to obtain a function $\Theta_t(\varphi_t) : \tilde K \to C_c(G_t)$.
Fixing $y \in \tilde K$ and letting $t \in [0,1]$ vary, we obtain a function $[\Theta(\varphi)](k) \in C_c(\mathcal G)$ defined by $[\Theta(\varphi)](k)(g) = [\Theta_t( \varphi_t)](k)(g)$ for all $g \in G_t$ and $t \in [0,1]$. Thus we obtain a function $\Theta(\varphi) : \tilde K \to C_c(\mathcal G)$. The twisted action of $(\tilde K,K)$ on $C^*_{\lambda}(G_t)$ induces a twisted action on $C^*_{\lambda}(\mathcal G)$ in a fiber-wise fashion. We thus obtain the following version of Lemma~\ref{twisted crossed} for the continuous field.

\begin{lem} The map $\Theta : C_c(\mathcal {\tilde G}) \to (\tilde K,K) \ltimes C^*_{\lambda}(\mathcal G)$ described above extends to an isomorphism
\[ C^*_{\lambda}(\mathcal{\tilde G}) \overset\cong\longrightarrow (\tilde K,K) \ltimes C^*_{\lambda}(\mathcal G). \]
For each $F$-orbit $\cO \subseteq \widehat{K}$, this isomorphism induces an isomorphism
\[ p_{\cO}\mathcal{\tilde C}_{\cO}p_{\cO} \overset\cong\longrightarrow (\tilde K,K) \ltimes p_{\cO}\mathcal Cp_{\cO} \]
that maps the subalgebra $p_{\cO}\mathcal{\tilde C}_{\cO}p_{\cO}$ of $\mathcal{\tilde C}_{\cO}$ onto the subalgebra $(\tilde K,K) \ltimes p_{\cO}\mathcal{C}_{\cO}p_{\cO}$ of $(\tilde K,K) \ltimes \mathcal{C}_{\cO}$. 

\end{lem}
\qed

Along the lines of Definition~\ref{transform}, we can define for each $\varphi \in C^*_{\lambda}(\mathcal G)$ a function
\[ \widehat{\varphi}_{\tau} : X_{\tau} \times [0,1] \to \text{End}(V_{\tau}) \]
using the operators $\pi( \varphi_t )$, $\pi \in X_{\tau}$ being identified with an irreducible representation of $\widehat{(G_t)}_{\lambda}$, $t \in [0,1]$, and $\varphi_t$ denoting the restriction of $\varphi$ to $G_t$. As in the proof of Proposition~\ref{direct sum}, the $C^*$-algebra $p_{\tau}\mathcal{C}^Kp_{\tau}$ is commutative, and it follows from Higson's result Theorem~\ref{higson constant} for the connected case that it is isomorphic to $C_0(X_{\tau} \times [0,1])$. Along the lines of Definition~\ref{twisted action}, we can define a twisted action of $(\tilde K, K)$ on the direct sum
\[ \bigoplus_{\tau \in \cO} C_0(X_{\tau} \times [0,1], \text{End}(V_{\tau}) )  \]
using a trivial action on $[0,1]$. 
We obtain the following versions of Propositions~\ref{direct sum} and \ref{equivariant} and Theorem~\ref{extended morita} for the continuous field.
\begin{prop}
Let $\cO$ be an orbit in $\widehat{K}$ under the action of $F$. Then the map $\varphi \mapsto \sum_{\tau \in \cO} \widehat{\varphi}_{\tau}$ establishes a $(\tilde K,K)$-equivariant isomorphism
\[ p_{\cO}\mathcal{C}_{\cO}p_{\cO} \overset\cong\longrightarrow \bigoplus_{\tau \in \cO} C_0(X_{\tau} \times [0,1], \text{End}(V_{\tau})). \]
\end{prop}
\qed

\begin{thm}\label{field morita} For every $F$-orbit $\cO \subseteq \widehat{K}$, the subquotient $\mathcal{\tilde C}_{\cO}$ is Morita equivalent to the twisted crossed product
\[ (\tilde K,K) \ltimes \bigoplus_{\tau \in \cO} C_0(X_{\tau} \times [0,1], \text{End}(V_{\tau})). \]
\end{thm}
\qed

The continuous field $\{ C^*_{\lambda}( \mathcal{\tilde G} ) \ | t \in [0,1] \}$ is now observed to be $K$-theoretically trivial in the sense below. As mentioned, this implies that $\tilde G$ satisfies the Baum-Connes conjecture with trivial coefficients. At this point the reasoning becomes identical to that of Higson in the connected case.

\begin{thm}\label{k theory}
For every $t \in [0,1]$, the homomorphism $C^*_{\lambda}(\mathcal{\tilde G} ) \to C^*_{\lambda}(\tilde G_t)$
given by evaluation of sections at $t$ induces an isomorphism
\[ K_*(C^*_{\lambda}(\mathcal{\tilde G})) \overset\simeq\longrightarrow K_*(C^*_{\lambda}(\tilde G_t)). \]
of $K$-theory groups.
\end{thm}
\begin{proof}
As detailed in \cite[Section 7]{MR2391803}, it suffices to establish that the homomorphism $C^*_{\lambda}(\mathcal{\tilde G}) \to C^*_{\lambda}(\tilde G_t)$ is an isomorphism for $t=1$, or indeed for any $t \in (0,1]$. Denoting by $\tilde C_{\cO}$ as in the previous section the subquotient of $C^*_{\lambda}(\tilde G_1) = C^*_{\lambda}(\tilde G)$ defined in Definition~\ref{orbit multipliers}, evaluation of sections at $t=1$ induces a homomorphism $\mathcal{\tilde C}_{\cO} \to \tilde C_{\cO}$.
As described by Higson, according to basic properties of $K$-theory it is sufficient for us to prove that this homomorphism induces an isomorphism of $K$-theory groups. Now the invariance of $K$-theory under Morita equivalence and Theorem~\ref{field morita} lead us to consider the commutative diagram
\[ \xymatrix{ (\tilde K,K) \ltimes p_{\cO}\mathcal{C}_{\cO}p_{\cO} \ar[d]\ar[r] & (\tilde K,K) \ltimes p_{\cO}C_{\cO}p_{\cO} \ar[d] \\
(\tilde K,K) \ltimes \bigoplus_{\tau \in \cO} C_0(X_{\tau} \times [0,1], \text{End}(V_{\tau})) \ar[r] & (\tilde K,K) \ltimes \bigoplus_{\tau \in \cO}  C_0(X_{\tau}, \text{End}(V_{\tau})).   } \]
The downward maps are isomorphisms and the left to right maps are induced by evaluation at $1 \in [0,1]$. Thus we need only show that bottom evaluation homomorphism induces an isomorphism of $K$-theory groups, and this follows from the homotopy invariance of $K$-theory.
\end{proof}

\bibliographystyle{alpha}
\bibliography{BIB}

\end{document}